\documentclass{amsart}[12pt]
\usepackage{amssymb,amsmath}
\usepackage{enumerate}
\usepackage[english]{babel}
\usepackage{color}

\newcommand{\Om} {\Omega}
\newcommand {\ep} {\epsilon}
\newcommand {\om} {\omega}
\newcommand {\gm} {\gamma}
\newcommand {\ii} {\infty}
\newcommand {\dt} {\delta}
\newcommand {\al} {\alpha}
\newcommand {\bt} {\beta}
\newcommand {\lb} {\lambda}
\newcommand {\Lb} {\Lambda}

\newcommand {\ol} {\overline}
\newcommand {\sm} {\setminus}
\newcommand {\su} {\subset}
\newcommand {\wt} {\widetilde}
\newcommand {\wh} {\widehat}

\newcommand {\cal} {\mathcal}
\newcommand {\mb} {\mathbf}

\newtheorem{teo}{Theorem}[section]
\newtheorem{pro}{Proposition}[section]
\newtheorem{cor}{Corollary}[section]

\newtheorem{lm}{Lemma}[section]

\theoremstyle{definition}
\newtheorem{rem}{Remark}[section]

\title{Almost uniform and strong convergences in ergodic theorems for symmetric spaces}

\begin{document}
\date{February 19, 2018}

\begin{abstract}
Let $(\Om,\mu)$ be a $\sigma$-finite measure space, and let $X\su L^1(\Om)+ L^\ii(\Om)$ be a fully symmetric space of measurable functions on $(\Om,\mu)$. If $\mu(\Om)=\ii$, necessary and sufficient conditions are given for almost uniform convergence
in $X$ (in Egorov's sense) of Ces\`aro averages $M_n(T)(f)=\frac1n \sum_{k = 0}^ {n-1} T^k(f)$ for all Dunford-Schwartz operators $T$ in $L^1(\Om)+ L^\ii(\Om)$ and any $f\in X$. Besides, it is proved that  the averages
$M_n(T)$ converge strongly in $X$ for each Dunford-Schwartz operator $T$ in $L^1(\Om)+ L^\ii(\Om)$
if and only if $X$ has order continuous norm and $L^1(\Om)$ is not contained in $X$.
\end{abstract}

\author{V. Chilin}
\address{National University of Uzbekistan\\
Tashkent, 100174, Uzbekistan}
\email{vladimirchil@gmail.com; chilin@ucd.uz}
\author{S. Litvinov}
\address{Pennsylvania State University \\
76 University Drive \\ Hazleton, PA 18202, USA}
\email{snl2@psu.edu}

\maketitle

\section{Introduction}
Let $(\Om,\cal A,\mu)$ be a complete $\sigma$-finite measure space.
Denote  by $L^0 =  L^0(\Omega)$ the algebra of equivalence classes of almost everywhere (a.e.)
finite real-valued measurable functions on $\Om$. Let $L^p=L^p(\Om)\su L^0$, $1\leq p\leq \ii$, be the
$L^p$-space equipped with the standard norm $\| \cdot \|_p$.

Let $T: L^1+ L^\ii \to  L^1+ L^\ii$ be a Dunford-Schwartz operator (writing $T\in DS$), that is, $T$ is linear and
$$
\| T(f)\|_1\leq \| f\|_1 \text{\ \ for all\ \ } f\in  L^1 \text{ \ \ and \ \ } \| T(f)\|_\ii\leq \| f\|_\ii \text{\ \ for all\ \ }
f\in \ L^\ii.
$$
 If $T\in DS$ is positive, that is, $T(f)\ge 0$ whenever $f\ge 0$, then we shall write $T\in DS^+$.

Dunford-Schwartz individual ergodic theorem  states that for every  $T \in DS$ and $f \in L^p$,
$1 \leq p <\ii$, the averages
\begin{equation}\label{e11}
M_n(T)(f)=\frac1n \sum_{k = 0}^ {n-1} T^k(f)
\end{equation}
converge a.e. to some $\widehat f\in L^p$  (see, for example, \cite[Ch.VIII, \S 6, Theorem VIII.6.6]{ds}).

In the case $\mu(\Om) <\ii$, it follows from Egorov's theorem that a.e. convergence coincides with the almost uniform (a.u.) convergence, thus, Dunford-Schwartz individual ergodic theorem asserts a.u. convergence of the averages $M_n(T)(f)$
for each $f \in L^p$, $1\leq p <\ii$, and all $T\in DS$.
If $\mu(\Om) = \ii$, then it is clear that a.u. convergence (in Egorov's sense) is generally stronger than a.e. convergence.

If $\mu(\Om)=\ii$, then there is the problem of describing the largest subspace of $ L^1+ L^\ii$ for which a.u. convergence in Dunford-Schwartz individual ergodic theorem holds. Let
\begin{equation}\label{e12}
\cal R_\mu=  \{f \in  L^1 +L^\ii:\mu\{|f|> \lambda\}<\ii \ \ \text{for all} \ \ \lambda>0\}.
\end{equation}
In Section 3 we prove (Theorem \ref{t31}) that for each $f\in\cal R_\mu$ and any $T \in DS$ the averages $M_n(T)(f)$
converge a.u. to some $\wh f\in\cal R_\mu$.  It should be pointed out that, by virtue of Lemma \ref{l31},
the proof of a.u. convergence in Dunford-Schwartz ergodic theorem is noticeably simpler than the proof of
a.e. convergence. We also show that if $(\Om,\cal A,\mu)$ has finitely many atoms
or if its atoms have equal measures, then $\cal R_\mu$ is the largest subspace of $L^1+ L^\ii$ for which
the convergence takes place: if $f\in (L^1+ L^\ii) \sm \cal R_\mu$, then there exists $T\in DS $ such that the sequence
$\{M_n(T)(f)\}$ does not converge a.u. (Theorem \ref{t34}).

A well-known mean ergodic theorem asserts (see, for example, \cite [Ch.VIII, \S 5] {ds}) that the averages
$M_n(T)$ converge strongly in a reflexive Banach space $(X,\|\cdot\|_{X})$
for every linear contraction $T$ of $X$, that is, given $x\in X$, there exists $\widehat x\in X$ such that
$$
\bigg \|\frac1n\sum_{k=0}^{n-1} T^k(x)-\widehat x\bigg\|_X \to 0 \ \ \text{as} \ \  n \to \ii.
$$
Important examples illustrating this ergodic theorem are the reflexive spaces $ L^p$, $1 <p <\ii$. In particular, the averages $M_n(T) $ converge strongly in $ L^p$ for any $T \in DS$. For the spaces
$L^1$ and $L^\ii$, the mean ergodic theorem is false, in general.

It is known that if $T\in DS$, then $T(X) \su X$ for any exact interpolation (for the Banach pair
$(L^1, L^\ii)$) symmetric space $X$ of real-valued measurable functions on $(\Om,\mu)$. In addition,
$ \| T \|_{X\to X}\leq 1$ (see, for example, \cite [Ch.II, \S 4, Sec.2] {kps}).
Recall also that the class of exact interpolation symmetric spaces for the Banach pair $ (L^1, L^\ii)$ coincides with the class
of fully symmetric spaces on $(\Om,\mu)$ \cite [Ch.II, \S 4, Theorem 4.3] {kps}. Therefore, there is the problem of describing the class of fully symmetric spaces $X$ for which the mean ergodic theorem with respect to the action of an arbitrary $T \in DS$ is valid.

If $X$ is a separable symmetric space on the non-atomic measure space $((0, a), \nu)$, where $0<a<\ii$ and $\nu$ is the Lebesgue measure, then the averages $M_n(T)$ converge strongly in $X$ for every $T\in DS$
(see \cite{ve, vf1}; also \cite[Ch.2, \S 2.1, Theorem 2.1.3]{vf}). At the same time, if $X$ is a non-separable fully
symmetric space, then for each $f\in X\sm\ol{L^\ii(0,a)}^{\|\cdot\|_ X}$ there are $T\in DS$ and a function
$\widetilde f$, equimeasurable with $f$, such that the sequence $\{A_n(T)(\wt f)\}$ does not converge strongly in $X$ \cite{vf1}. Note also that, for the separable symmetric space $L^1((0, \ii),\nu)$, there exists $T \in DS$
such that the averages $M_n(T)$ do not converge strongly in $L^1((0, \ii), \nu) $.

The main result of Section 4 is Theorem \ref{t45}, which gives a criterion for a fully symmetric space $X$
to satisfy the following: the averages $M_n(T)(f)$ converge strongly in $X$ for every $f\in X$ and $T\in DS$.

In Section 5 we discuss some (classes of) fully symmetric spaces for which Dunford-Schwartz-type ergodic theorems hold/fail.

\section{Preliminaries}
Let $(\Om,\cal A,\mu) $ be a $\sigma$-finite measure space and let $L^0 = L^0(\Om)$ be the algebra of (classes of) a.e. finite real-valued measurable functions on $(\Omega, \mathcal A, \mu) $.
Let $L^0_\mu$ be the subalgebra in $L^0$ consisting of the functions $f\in L^0 $ such that
$\mu\{| f |>\lb \}<\ii$ for some $\lb>0$.

In what follows $t_\mu$ will stand for the {\it measure topology} in $L^0$, that is, the topology given by the following system of neighborhoods of zero:
$$
\cal N(\ep,\dt)=\{ f\in L^0: \ \mu\{|f|>\dt\}\leq\ep\}, \ \ \ep>0, \ \dt>0.
$$
It is well-known (see, for example, \cite[Ch.IV, \S 27, Theorem 5]{ha}) that $(L^0,t_\mu)$ is a complete metrizable vector space. Since $L^0_\mu$ is a closed linear subspace of $(L^0,t_\mu)$ (see, for example, \cite[Ch.II, \S 2]{kps}),
$(L^0_\mu,t_\mu)$ is also a complete metrizable vector space.

Denote by $L^p=L^p(\Om)\su L^0_\mu$,
$1 \leq p \leq \ii$, the classical Banach space equipped with the standard norm $\|\cdot\|_p$.

If $f\in L^0_\mu$, then a {\it non-increasing rearrangement} of $f$ is defined as
$$
f^*(t)=\inf \{\lambda>0: \ \mu\{|f| > \lambda\} \leq t\}, \ \  t>0,
$$
(see \cite[Ch.II, \S 2]{kps}).

Consider the $\sigma$-finite measure space $((0,\ii), \nu)$, where $\nu$ is the Lebesgue measure.
A non-zero linear subspace  $X\su L^0_\nu$ with a Banach norm $\|\cdot\|_X$ is called
{\it symmetric (fully symmetric)} on  $ ((0,\ii), \nu) $ if
$$
f \in X, \ g\in L^0_\nu, \ g^*(t)\leq f^*(t) \text{\ \ for all \ \ } t>0
$$
(respectively,
$$
f \in X, \ g \in L^0_\nu, \ \int\limits_0^s g^*(t) dt\leq\int\limits_0^s f^*(t) dt\text{\ \ for all\ \ } s>0
\ \ (\text{writing\ } g \prec\prec f)
$$
implies that $g \in X$ and $\| g\|_X\leq \| f\|_X$.

Let $(X,\| \cdot \|_X)$ be a symmetric (fully symmetric) space on $((0,\ii),\nu)$.
Define
$$
X(\Omega)=\{ f\in L^0_\mu:  f^*(t) \in X\}
$$
and set
$$
\| f\|_{X(\Om)}=\|  f^*(t)\|_X,  \ f\in X(\Om).
$$
It is shown in \cite{ks} (see also \cite[Ch.3, Sec.3.5]{lsz}) that $(X(\Om), \| \cdot \|_{X(\Om})$
is a Banach space and conditions  $f \in X(\Om)$, $g \in L^0_\mu$, $g^*(t)\leq f^*(t)$ for every $t>0$
(respectively,  $g \prec\prec f)$ imply that $g \in X(\Om)$ and $\| g\|_{X(\Om)}\leq \| f\|_{X(\Om)}$. In such a case,
we say that $(X(\Om), \|\cdot\|_{X(\Om)})$ is a symmetric (fully symmetric) space on  $(\Omega, \mathcal A, \mu)$
{\it generated} by the symmetric (fully symmetric) space $(X, \|\cdot\|_X)$.

Immediate examples of fully symmetric spaces are $ L^1(\Om)\cap  L^\ii(\Om)$  with the norm
$$
\|f\|_{L^1\cap L^\ii}=\max \left \{ \|f\|_{L^1(\Om)}, \|f\|_{L^\ii(\Om)} \right\}
$$
and $ L^1(\Om) +  L^\ii(\Om)$ with the norm
$$
\|f\|_{L^1 +  L^\ii}=\inf \left \{ \|g\|_1+ \|h\|_\ii: \ f = g + h, \ g \in L^1(\Om), \ h \in \ L^\ii(\Om) \right \}=
\int_0^1 f^*(t) dt
$$
(see \cite[Ch.II, \S 4]{kps}).

Every fully symmetric  space $X=X(\Om)$ is an exact interpolation space for the Banach couple
 $(L^1(\Om), L^\ii(\Om))$, hence $T(X) \su X$ and  $\|T\|_{X\to X} \leq 1$ for any $T\in DS$  
(see, for example, \cite[Ch.3, \S 2, Theorem 2.2]{bs}, \cite[Ch.II, \S 4, Theorem 4.3]{kps}).

For any symmetric space $X(\Om)$ the following continuous embeddings hold:
$$
L^1(\Om)\cap L^\ii(\Om)\su X(\Om)\su L^1(\Om)+ L^\ii(\Om)
$$
\cite[Ch.2, \S 6, Theorem 6.6]{bs}.

A symmetric space $(X(\Om), \| \cdot \|_{X(\Om)})$ is said to have {\it order continuous norm} if
$\| f_\al\|_{X(\Om)}\downarrow 0$ whenever $f_{\al}\in X$ and $f_{\al}\downarrow 0$.
It is clear that a symmetric space $(X(\Om), \| \cdot \|_{X(\Om)})$, that is generated by a symmetric space $(X, \|\cdot\|_X)$, has (respectively, has no) order-continuous norm if and only if a space $(X, \|\cdot\|_X)$  also has (or does not have) an order-continuous norm.
Besides, a symmetric space $(X,\|\cdot\|_X)$ on $((0,\ii),\nu)$ has order-continuous norm if and only if it is separable \cite[II, Ch.6, \S 6.5, Theorem 6.5.3.]{rgmp}. In addition, $(X, \| \cdot \|_X)$ is a fully symmetric space \cite[Ch.II, \S 4, Theorem 4.10]{kps}.

If $(X, \|\cdot\|_X)$ is a symmetric space on $ ((0,\ii), \nu) $, then the K\"othe dual $X^\times$  is defined as
$$
X^\times=\{f \in L^0_\nu: \ fg \in L^1 \text{ \ for all \ } g\in X\},
$$
and
$$
\|f\|_{X^\times}=\sup\left \{ \bigg |\int\limits_0^\infty fg \ d \nu \bigg|:\ g \in X, \ \|g\|_X\leq 1\right\}
\text{\ \ if\ \ }f \in X^\times.
$$
It is known that $(X^\times, \|\cdot\|_{X^\times})$ is a fully symmetric space (see, for example, \cite[Ch.II, \S 4, Theorem 4.9]{kps}, \cite[II, Ch.7, \S 7.2,  Theorem 7.2.2]{rgmp}.
In addition,
$$
X \su X^{\times\times}, \  (L^1)^\times =  L^\ii, \ (L^\ii)^\times = L^1;
$$
$$
(L^1 + L^\ii,\|\cdot\|_{L^1 + L^\ii})^\times = (L^1 \cap L^\ii,\|\cdot\|_{L^1 \cap L^\ii});
$$
$$
(L_1 \cap L^\ii,\|\cdot\|_{L^1 \cap L^\ii})^\times = (L^1 + L^\ii,\|\cdot\|_{L^1 + L^\ii})
$$
(see \cite[II, Ch.7]{rgmp}).

Note that
$$X^\times(\Om) =\{f \in L^0_\mu: \ fg \in L^1(\Om) \text{ \ for all \ } g \in X(\Om)\},$$
and
$$
\|f\|_{X^\times(\Om)}=\sup\left \{ \bigg |\int\limits_\Om fg \ d \mu \bigg|:\ g \in X(\Om), \ \|g\|_{X(\Om)} \leq 1\right\}, \ f \in X^\times(\Omega).
$$
A  fully symmetric space $(X^\times(\Om), \|\cdot\|_{X^\times(\Om)})$ is called the K\"othe dual space for symmetric
space $(X(\Om), \|\cdot\|_{X(\Om)})$.

A symmetric space $(X, \| \cdot \|_X)$ is said to possess {\it Fatou property} if conditions
$$
0 \leq f_n\in X, \ \ f_n\leq f_{n+1} \text{ \ for all\ }n, \text{ \ and\ }\sup_n \| f_n\|_X<\ii
$$
imply that there exists $f=\sup \limits_nf_n\in X$ and $\| f\|_X=\sup \limits_n \| f_n\|_X$.

If $X = X^{\times\times}$, then the symmetric space $X$ possesses Fatou property (see, for example, 
\cite[Vol.II, Ch.I, \S 1b]{lt}); in particular, the fully space $(L^1 + L^\ii,\|\cdot\|_{L^1 + L^\ii})$ possesses  Fatou property.
In addition, in any symmetric space $(X, \| \cdot \|_X)$ with Fatou property conditions $f_n \in X$, 
$\sup_n\| f_n\|_X\leq\al$, $f \in L^0$ and $f_n\to f$ in $t_\mu$ imply that $f\in X$ and $\|f\|_X\leq \al$
(see, for example, \cite[Ch.IV, \S 3, Lemma 5]{ka}).

Define
$$
\cal R_\mu= \{f \in  L^1(\Om)+L^\ii(\Om): \ f^*(t) \to 0 \text{ \ as \ } t\to \ii\}.
$$
It is clear that $\cal R_\mu$ admits a more direct description (\ref{e12}).

Note that if $\mu(\Om) < \ii$, then $\cal R_\mu$ is simply $ L^1(\Omega)$.
Therefore, we will be concerned with infinite measure spaces.

By \cite[Ch.II, \S 4, Lemma 4.4]{kps}, $(\cal R_\mu, \|\cdot\|_{ L^1 +L^\ii})$ is a symmetric space. In addition,
$\cal R_\mu$ is the closure of $L^1(\Omega)\cap  L^\ii(\Omega)$  in $ L^1(\Omega) +  L^\ii(\Omega)$ (see \cite[Ch.II, \S 3, Section 1]{kps}). In particular, $(\cal R_\mu, \|\cdot\|_{L^1 +  L^\ii})$ is a fully symmetric space.

Let $\chi_E$ be the characteristic function of a set $E\in\cal A$. Denote $\mb 1 = \chi_\Om$. The following
gives a necessary and sufficient condition for the embedding of a symmetric space into $\cal R_\mu$.

\begin{pro}\label{p21}
If $\mu(\Om)=\ii$, then a symmetric space $X=X(\Om)\su L^0_\mu$ is contained in
$\cal R_\mu$ if and only if $\mb 1\notin X$.
\end{pro}
\begin{proof}
As  $\mu(\Om) = \infty$, we have $\mb 1^*(t) = 1$ for all $t > 0$, hence $\mb 1 \notin \cal R_\tau$. Therefore
$X$ is not contained in $\cal R_\mu$ whenever $\mb 1 \in X$.

Let  $\mb 1\notin X$. If $f \in X$ and $\lim\limits_{t\to \ii} f^*(t) =\al>0$, then
$$
\mb 1^*(t) \equiv 1 \leq \frac1\al f^*(t),
$$
implying $\mb 1\in  X$, a contradiction. Thus $\mb 1\notin X$ entails $X \su\cal R_\mu$.
\end{proof}

We will also need the next property of the fully symmetric space $\cal R_\mu$.

\begin{pro}\label{p22}
For every $f\in\cal R_\mu$ and $\ep>0$ there exist $g_\ep\in  L^1$  and $h_\ep\in  L^\ii$ such that
$$
f=g_{\ep}+h_{\ep} \text{ \ and \ } \| h_\ep\|_\ii\leq \ep.
$$
\end{pro}

\begin{proof} If
$$
\Om_\ep=\{|f|>\ep\}, \ \ g_\ep=f \cdot \chi_{\Om_\ep}, \ \ h_\ep=f \cdot \chi_{\Om \sm \Om_\ep},
$$
then $\| h_{\ep}\|_{\ii}\leq \ep$. Besides, as $f\in L^1+ L^\ii$, we have
$$
f=g_\ep+h_\ep=g+h
$$
for some $g\in L^1$, $h\in L^\ii$. Then, since $f\in \cal R_\mu$, we have $\mu(\Om_\ep)<\ii$, which implies that
$$
g_{\ep}=g\cdot \chi_{\Om_\ep}+(h-h_\ep)\cdot \chi_{\Om_\ep}\in L^1.
$$
\end{proof}

\section{Almost uniform convergence in Dunford-Schwartz pointwise ergodic theorem}

The main goal of this section is to prove the following extension of the classical Dunford-Schwartz pointwise ergodic theorem.

\begin{teo}\label{t31}
Assume that $(\Om, \cal A, \mu)$ is an arbitrary measure space, and let $X$ be a fully symmetric space on
$(\Om,\cal A, \mu)$ such that $\mb 1\notin X$. If $T\in DS$ and $f\in X$, then the averages (\ref{e11}) converge a.u. to some $\wh f\in X$.  In particular, $M_n(T)(f)\to \wh f\in \cal R_\mu$ a.u. for all $f\in \cal R_\mu$.
\end{teo}

\begin{rem}
In proving Theorem \ref{t31}, we can and will assume that $(\Om,\cal A,\mu)$  is $\sigma$-finite. Indeed, if $f\in X$ and
$\mb 1\notin X$, then $f\in \cal R_\mu$ by Proposition \ref{p21}, which implies that $\{T^k(f)\}_{k=0}^\ii\su \cal R_\mu$.
Therefore, the set $\Om_f=\bigcup\limits_k\ \{T^k(f)\neq 0\}$ is $\sigma$-finite, and one can replace $\Om$ by $\Om_f$.
\end{rem}

In view of Propositions \ref{p21} and \ref{p22}, proof of Theorem \ref{t31} can be easily reduced to the case $X=L^1$,
so we shall treat this case first.

Let $(X, \|\cdot\|_X)$ be a Banach space, and let $M_n: X\to L^0$ be a sequence of linear maps. Denote
$$
M^\star(f)=\sup_n|M_n(f)|,
$$
the {\it maximal function} of $f\in X$. If $M^\star(f)\in L^0$ for all $f\in X$, then the function
$$
M^\star: X\to L^0, \ \ f\in X,
$$
is called the {\it maximal operator} of the sequence $\{M_n\}$.

\begin{rem}\label{r31}
(1) If $\mu(\Om)<\ii$, then the Banach principle implies that if $M^\star(f)\in L^0$ for all $f\in X$, then
$M^\star: (X,\|\cdot\|_X)\to (L^0,t_\mu)$ is continuous at zero, which is not the case when $\mu$ is not finite;
see \cite[Sec.2]{li1}.

\noindent
(2) If $f\in X$ and $E\in \cal A$ are such that $\{M_n(f)\chi_E\}\su L^\ii$, then it is easy to see that
$$
\|M^\star(f)\chi_E\|_\ii=\sup_n\|M_n(f)\chi_E\|_\ii
$$
\cite[Proposition 1.1]{li}. Therefore, the continuity of $M^\star: (X,\|\cdot\|_X)\to (L^0,t_\mu)$ at zero can be expressed as follows: given $\ep>0$, $\dt>0$, there is $\gm>0$ such that for every $f\in X$ with $\| f\|_X \leq\gm$
it is possible to find $E\su \Om$ satisfying conditions
$$
\mu(\Om\sm E)\leq\ep \text{ \ \ and \ \ } \sup_n\|M_n(f)\chi_E\|_\ii\leq\dt.
$$
\end{rem}

A sequence $\{f_n\}\su L^0$ is said to converge {\it almost uniformly (a.u.)} to $f\in L^0$ if for every $\ep>0$
there is a set $E\su \Om$ such that $\mu(\Om\sm E)\leq\ep$ and $\|(f-f_n)\chi_E\|_\ii\to 0$.

\begin{pro}\label{p31}
The algebra $L^0_\mu$ is complete with respect to a.u. convergence.
\end{pro}
\begin{proof}
Assume that $\{f_n\}\su L^0_\mu$ is a.u. Cauchy. Then it is clearly Cauchy with respect to $t_\mu$. Since $L^0_\mu$
is complete relative to $t_\mu$, there is $f\in L^0_\mu$ such that $f_n\to f$ in measure.

To show that $f_n\to f$ a.u., fix $\ep>0$. Since $f_n\in L^0_\mu$ for every $n$ and $\{ f_n\}$ is a.u. Cauchy, it is possible
to construct $E\su \Om$ such that $\mu(\Om\sm E)\leq\ep$, $f_n\chi_E\in L^\ii$ for every $n$, and
$$
\|(f_m-f_n)\chi_E\|_\ii\to 0 \text{ \ \ as\ \ } m, n\to\ii.
$$
This implies that there is $\wh f\in L^\ii$ such that $\| \wh f-f_n\chi_E\|_\ii\to 0$, hence $f_n\chi_E\to \wh f$
in measure. But $f_n\to f$ in measure implies that $f_n\chi_E\to f\chi_E$ in measure, hence $\wh f=f\chi_E$ and so
$$
\|(f-f_n)\chi_E\|_\ii=\| \wh f-f_n\chi_E\|_\ii\to 0.
$$
Therefore, the sequence $\{f_n\}$ is a.u. convergent in $L^0_\mu$, that is, $L^0_\mu$ is a.u. complete.
\end{proof}

\begin{lm}\label{l31}
If the maximal operator $M^\star: (X,\|\cdot\|_X)\to (L^0,t_\mu)$ of a sequence $M_n: X\to L^0_\mu$ of linear maps is
continuous at zero, then the set
$$
X_c= \{ f\in X: \{M_n(f)\} \text{\ converges a.u.}\}
$$
is closed in $X$.
\end{lm}
\begin{proof}
Let $X_c\ni f_k\to f$ in the norm $\|\cdot\|_X$. Fix $\ep>0$, $\dt>0$. In view of Remark \ref{r31}(2) and since $M^\star$
is continuous at zero, there are $f_{k_0}$ and $E_0\su \Om$ such that
$$
\mu(\Om\sm E_0)\leq\frac\ep2 \text{ \ \ and \ \ } \sup_n\|M_n(f-f_{k_0})\chi_{E_0}\|_\ii\leq\frac\dt3.
$$
Next, since the sequence $\{M_n(f_{k_0})\}$ converges a.u., there are $E_1\su \Om$ and $N\in \mathbb N$ such that
$$
\mu(\Om\sm E_1)\leq\frac\ep2 \text{ \ \ and \ \ }
\|(M_m(f_{k_0})-M_n(f_{k_0}))\chi_{E_1}\|_\ii\leq\frac\dt3 \text{ \ \ for all\ \ } m,n\ge N.
$$
Therefore, setting $E=E_0\cap E_1$, we arrive at $\mu(\Om\sm E)\leq\ep$ and
\begin{equation*}
\begin{split}
\|(M_m(f)-M_n(f))\chi_E\|_\ii & \leq \|M_m(f-f_{k_0})\chi_E\|_\ii+\|M_n(f-f_{k_0})\chi_E\|_\ii \\
&+\|(M_m(f_{k_0})-M_n(f_{k_0}))\chi_E\|_\ii\leq \dt
\end{split}
\end{equation*}
for all $m,n\ge N$. This means that the sequence $\{M_n(f)\}$ is a.u. Cauchy, which, by Proposition \ref{p31}, entails
that $\{M_n(f)\}$ converges a.u., hence $f\in X_c$, and we conclude that $X_c$ is closed in $(X,\|\cdot\|_X)$.
\end{proof}

Therefore, since $T(L^1)\su L^1\su L^0_\mu$ for a given $T\in DS$, in order to prove that the averages (\ref{e11})
converge a.u. for every $f\in X=L^1$, it is sufficient to show that
\begin{enumerate}[(A)]
\item
the maximal operator
$M(T)^\star: (L^1, \|\cdot\|_1)\to (L^0,t_\mu)$ is continuous at zero;
\item
there is a dense subset $\cal D$ of $L^1$
such that the sequence $\{M_n(T)(f)\}$ converges a.u. whenever $f\in \cal D$.
\end{enumerate}

Here is our main tool, Hopf's maximal ergodic theorem \cite{ho}; see also \cite[Theorem 1.1, p.75]{pe}:

\begin{teo}\label{t32}
If $T: L^1\to  L^1$ is a positive linear contraction
and $f \in  L^1$, then
$$
\int\limits_{\{M(T)^\star(f)>0\}} f \ d\mu \ge 0.
$$
\end{teo}

We shall prove the following maximal inequality for $T\in DS$ acting in $L^p$, $1\leq p<\ii$. Note that,
in order to establish Theorem \ref{t31}, we will only need it for $p=1$ and $p=2$.

\begin{teo}\label{t33}
If $T\in DS$ and $1\leq p<\ii$,  then
\begin{equation}\label{e31}
\mu \{M(T)^{\star}(|f|)>\lb \} \leq \left (2\frac {\| f\|_p}{\lb} \right )^p \text{ \ \ for all\ \ } f\in L^p, \ \lb >0.
\end{equation}
\end{teo}
\begin{proof}
Assume first that $T\in DS^+$. Fix $f\in L^1$ and $\lb>0$. Pick $F\su \Om$ such that $\mu(F) < \ii$ and let
$f_{\lb,F}=|f|-\lb\chi_F$. Then, since
$\lb\chi_F \in  L^\ii$, we have $\| T(\lb \chi_F)\|_\ii \leq \lb$, hence $T(\lb \chi_F)\leq \lb\cdot \mb 1$.
Therefore $T(f_{\lb,F})\ge T(|f|)- \lb \cdot \mb 1$, and we derive
$$
M(T)^\star(f_{\lb,F})\ge M(T)^\star(|f|)-\lb\cdot \mb 1.
$$
By Theorem \ref{t32},
$$
\int_{\{M(T)^\star(f_{\lb,F})>0\} }f_{\lb,F}\ge 0,
$$
implying that
$$
\begin{aligned}
 \| f\|_1
&\ge \int_{\{M(T)^\star(f_{\lb,F})>0\} } |f|=\int_{\{M(T)^\star(f_{\lb,F})>0\}} (f_{\lb,F}+\lb \chi_F) \\
&\ge \int_{\{M(T)^\star(f_{\lb,F})>0\} }\lb \chi_F = \lb \mu \left (\{M(T)^\star(f_{\lb,F})>0\}\cap F \right) \\
&\ge \lb \mu \left (\{M(T)^\star(|f|)>\lb \}\cap F\right ),
\end{aligned}
$$
\vskip 2pt
\noindent
Therefore, we have
$$
\mu \left (\{M(T)^\star(|f|)>\lb \}\cap F\right ) \leq \frac {\| f\|_1}\lb
$$
for every $F\su \Om$ with $\mu(F)<\ii$. Since  $(\Om, \mu)$ is a  $\sigma$-finite, we arrive at the following
maximal inequality for $T\in DS^+$ acting in $L^1$:
\begin{equation}\label{e32}
\mu \{M(T)^\star(|f|)>\lb \}\leq\frac {\| f\|_1}\lb \text{ \ \ for all\ \ }f\in  L^1, \ \lb>0.
\end{equation}

Now, fix $1<p<\ii$, $f\in L^p$, and $\lb>0$. Since $t\ge\frac \lb2$ implies $t\leq \left ( \frac 2\lb \right )^{p-1}t^p$,
we have
$$
|f(\om)|\leq  \left ( \frac 2\lb \right )^{p-1}|f(\om)|^p \text{ \ \ whenever \ } |f(\om)|\ge\frac\lb2.
$$
Then, denoting $A_\lb=\{|f|< \lb /2 \}$ and $g_\lb=|f|\chi_{A_\lb}$, we obtain
$$
|f|\leq g_\lb+\left ( \frac 2 \lb \right )^{p-1}|f|^p.
$$
Since $g_\lb\in  L^\ii$,
we have $\| T(g_\lb)\|_\ii\leq \| g_\lb\|_{\ii}\leq \frac \lb 2$, and it follows that
$$
M(T)^\star(|f|)\leq\frac \lb 2\cdot \mb 1+\left ( \frac 2\lb \right )^{p-1}M(T)^\star(|f|^p).
$$
As $|f|^p \in  L^1$, employing (\ref{e32}), we obtain a maximal inequality for $T\in DS^+$ acting in $L^p$,
$1\leq p<\ii$:

\begin{equation}\label{e33}
\begin{split}
\mu \{M(T)^\star(|f|)>\lb \}&
\leq \mu \left \{ \left (\frac 2 \lb\right )^{p-1}M(T)^\star(|f|^p)>\frac \lb 2 \right \} \\
&=\mu \left \{(M(T)^\star(|f|^p)>\left (\frac \lb2 \right )^p \right \} \\
&\leq \left ( 2\frac {\|f \|_p}\lb \right )^p, \ \ \ \ \ f\in L^p, \ \lb>0.
\end{split}
\end{equation}

Finally, let $T\in DS$. If $|T|:L^p\to  L^p$  is the linear modulus of $T:L^p\to  L^p$,
then $|T|\in DS^+$ and $|T^k(f)|\leq |T|^k(|f|)$
for all $f\in L^p$ and $1 \leq p \leq \ii$, $k=0,1,\dots$ (see, for example, \cite[Ch.1, \S 1.3]{me}, \cite[Ch.4, \S 4.1, Theorem 1.1]{kr}). Therefore, given $f\in L^p$, $1\leq p<\ii$,
we have
$$
|M_n(T)(|f|)|\leq\frac 1n \sum_{k=0}^{n-1}\left |T^k(|f|)\right |\leq \frac 1n \sum_{k=0}^{n-1} |T|^k(|f|) = M_n(|T|)(|f|).
$$
Thus, applying inequality (\ref{e33}) to $|T|\in DS^+$, we obtain (\ref{e31}):
$$
\mu\{M(T)^\star(|f|)>\lb\} \leq  \mu\{M(|T|)^\star(|f|)>\lb \}\leq\left ( 2\frac {\|f \|_p}\lb \right )^p, \ \ f\in L^p, \ \lb>0.
$$
\end{proof}

Now we can prove Theorem \ref{t31}:
\begin{proof}
Show first that the sequence $\{M_n(f)\}$ converges a.u. whenever $f\in L^1$. In view of Theorem \ref{t33},
the maximal operator $M(T)^\star: (L^p,\|\cdot\|)\to (L^0,t_\mu)$ is continuous at zero for
every $1\leq p<\ii$. This, by Lemma \ref{l31}, implies that the set
$$
\cal C_p=\{ f\in L^p: \{M_n(f)\} \text{\ converges a.u.} \}
$$
is closed in $L^p$, $1\leq p<\ii$.

In particular, the sets $\cal C_1$ and $\cal C_2$ are closed in $L^1$ and $L^2$, respectively. Therefore,
taking into account that the set $L^1\cap L^2$ is dense in $L^1$,
it is sufficient to show that the sequence $\{M_n(T)(f)\}$ converges a.u. for each $f$ in a dense subset $L^2$.

Denote by $(\cdot,\cdot)$ the standard inner product in $L^2$.  Let
$$
N=\{T(h) - h: \ h \in L^2\cap L^\ii\}.
$$
If $L^2 \ni g\in N^\perp$, then, as $L^2\cap L^\ii$ is dense in $L^2$, we have
$$
0=(g,T(h)-h)=(T^*(g)-g,h), \ \ h\in L^2,
$$
so $T^*(g)=g$. Recalling that $T$ is a contraction in $L^2$, we obtain
\begin{equation}\label{e34}
\begin{aligned}
\| T(g)-g\|^2_2
&=(T(g)-g,T(g)-g)= \\
&\| T(g)\|^2_2 -  (g,T^*(g))- (T^*(g),g)+\| g\|^2_2
&=\| T(g)\|^2_2-\| g\|^2_2\leq 0, \\
\end{aligned}
\end{equation}
so $T(g)=g$ as well, hence $N^{\perp}\su L=\{g\in L^2: T(g)=g\}$. Conversely, if $g\in L$, then,
since $T^*$ is also a contraction in $L^2$, replacing $T$ by $T^*$ in (\ref{e34}), we obtain
$T^*(g)=g$, which implies that $g\in N^\perp$. Therefore $N^\perp=L$, hence $\ol N \oplus L=L^2$,
and we conclude that the set
$$
\cal D=\{ g+(T(h)-h) :\  g\in L^2,\  T(g)=g; \ h\in L^2\cap L^\ii\}
$$
is dense in $L^2$. Because $h\in L^\ii$, it is clear that the
sequence $\{M_n(f)\}$ converges a.u. for every $f\in \cal D$, and we conclude that this sequence converges a.u. for all
$f\in L^1$.

Now, let $X \su L^1+L^\ii$ be a fully symmetric space such that $\mb 1\notin X$, and let $f\in X$.
By Proposition \ref{p21}, $f\in \cal R_\mu$.

Fix $\ep>0$ and $\dt>0$. In view of Proposition \ref{p22}, there are $g\in L^1$ and $h\in L^\ii$ such that
$$
f=g+h, \ \ g\in L^1, \text{\ \ and\ \ } \| h\|_\ii\leq\frac\dt3.
$$
Since $g\in L^1$, there exists $E\su \Om$ and $N\in \mathbb N$ satisfying conditions
$$
\mu(\Om\sm E)\leq\ep \text{ \ \ and \ \ } \| (M_m(g)-M_n(g))\chi_E\|_\ii\leq\frac\dt3 \text{ \ \ for all\ \ }m,n\ge N.
$$
Then, given $m,n\ge N$, we have
\begin{equation*}
\begin{split}
\|(M_m(f)-M_n(f))\chi_E\|_\ii&=\|(M_m(g)-M_n(g))\chi_E\|_\ii+\| (M_m(h)-M_n(h))\chi_E\|_\ii \\
&\leq \frac\dt3+\|M_m(h)\|_\ii+\|M_n(h)\|_\ii\leq \frac\dt3+2\|h\|_\ii\leq\dt,
\end{split}
\end{equation*}
implying, by Proposition \ref{p31}, \ that the sequence \ $\{M_n(f)\}$ \ converges a.u. to some \ $\wh f\in L^0_\mu$.

Since  $L^1+L^{\infty}$ possesses Fatou property and 
$$M_n(f)\in L^1+L^\ii, \ \ \sup_n \|M_n(f)\|_{L^1+L^\ii} \leq \|f\|_{L^1+L^\ii}, \ \ M_n(f)\to\wh f \text{ \ in \ } t_\mu,
$$
it follows that $\wh f \in L^1+L^\ii$ \cite[Ch.IV, \S 3, Lemma 5]{ka}.
In addition,  
$$
M_n(f)^*(t) \to \wh f^{\, *}(t) \text{ \ \ a.e.  on\ \ } (0,\ii)
$$
(see, for example, \cite[Ch.II, \S 2, Property $11^\circ$]{kps}).
Since $T \in DS$, it follows that $M_n(f)^*(t)\prec\prec f^*(t)$ for all $n$ (see, for example, 
\cite[Ch.II, \S 3, Section 4]{kps}). Consequently, by the Fatou Theorem,
$$
\int \limits_0^s\wh f^{\,*}(t) dt \leq \sup_n \int \limits_0^s M_n(f)^*(t) dt \leq \int\limits_0^s f^*(t) dt \text{ \ \ for all\ \  }s>0,
$$
that is, $\wh f^{\,*}(t)\prec\prec f^*(t)$.
Since $ X$ is a fully symmetric space and $f\in  X$, it follows that $\widehat f\in  X$.
\end{proof}

Now we shall present a class of measure spaces $(\Om,\cal A,\mu)$ for which $E=\cal R_\mu$ is the largest
fully symmetric subspace of $L^1+ L^\ii$ for which Theorem \ref{t31} is valid.
We say that a measure space  $(\Om,\cal A,\mu)$  is {\it quasi-non-atomic}
if it has finitely many atoms or its atoms have the same measure.

\begin{teo}\label{t34}
Let  an $\sigma$-finite infinite measure space $(\Om,\cal A,\mu)$ be quasi-non-atomic.
Given a fully symmetric space $X \subset  L^1+L^\ii$, the following conditions are equivalent:
\begin{enumerate}[(i)]
\item $X \subseteq \cal R_\mu$.
\item $\mathbf 1\notin X$.
\item For every $f \in X$ and $T \in DS$ the averages (\ref{e11}) converge a.u. to some $\wh f\in X$.
\end{enumerate}
\end{teo}
\begin{proof}
Implications (i) $\Longleftrightarrow$ (ii) is Proposition \ref{p21}, while
(ii) $\Longrightarrow$ (iii) is Theorem \ref{t31}.
Implication (iii) $\Longrightarrow$ (i) follows from \cite[Theorem 4.1]{cl}, where it was established that,
in the case of quasi-non-atomic $\sigma$-finite measure space $(\Om,\cal A,\mu)$,
for a given $f\in (L^1+L^\ii)\sm \cal R_\mu$, there exists $T\in DS$ such that the averages (\ref{e11})
do not converge a.e., hence do not converge a.u.
\end{proof}

Now we shall present some examples of fully symmetric spaces $X$ such that $\mb 1\notin X$ or $\mb 1\in X$.
Recall that it is assumed that $\mu(\Om)=\ii$.

1. Let $\Phi$ be an {\it Orlicz function}, that is, $\Phi:[0,\ii)\to [0,\ii)$ is left-continuous, convex,  increasing function such that $\Phi(0)=0$ and $\Phi(u)>0$ for some $u\ne 0$ (see, for example \cite[Ch.2, \S 2.1]{es}).  Let
$$
 L^\Phi=\left \{ f \in  L^0_\mu: \  \int_\Om \Phi\left (\frac {|f|}a \right )d \mu
<\ii \text { \ for some \ } a>0 \right \}
$$
be the corresponding {\it Orlicz space},  and let
$$
\| f\|_\Phi=\inf \left \{ a>0:  \int_\Om \Phi\left (\frac {|f|}a \right )d \mu \leq 1\right\}
$$
be the {\it Luxemburg norm} in $L^\Phi$. \ It is well-known that  $(L^\Phi, \| \cdot\|_\Phi)$
is a fully symmetric space.

Since  $\mu(\Om) = \infty$, if $\Phi(u)>0$ for all $u\ne 0$, then
$\int_\Om \Phi\left (\frac 1a \cdot\mb 1 \right ) d\mu = \ii$ for each
$a>0$, hence $\mb1\notin L^\Phi$. If  $\Phi(u)=0$ for all $0\leq u< u_0$, then  $\mb1 \in L^\Phi$.

2. If $X$  is a symmetric space  with order continuous norm, then
$\mu\{|f|>\lb\}<\ii$  for all $f \in X$ and $\lb> 0$, so  $X\su \cal R_\mu $; in particular, $\mb 1 \notin X$.

3.  Let $\varphi$ be a concave function on $[0, \infty)$ with $\varphi(0) = 0$ and
$\varphi(t) > 0$ for all $t > 0$, and let
$$
\Lambda_\varphi= \left \{f \in  L^0_\mu: \ \|f \|_{\Lambda_\varphi} =
\int_0^{\infty} f^*(t) d \varphi(t) < \infty \right \}
$$
be the corresponding {\it Lorentz space}.

It is well-known that $(\Lambda_\varphi, \|\cdot\|_{\Lambda_\varphi})$
is a fully symmetric space; in addition, if $\varphi(\infty) = \ii$, then
$\mathbf 1  \notin  \Lambda_\varphi$ and if $\varphi(\infty) < \ii$, then
$\mathbf 1  \in  \Lambda_\varphi$.

Let $\varphi$ be as above, and let
$$
M_\varphi=\left \{f \in L^0_\mu: \ \|f \|_{M_\varphi} =
\sup\limits_{0<s< \infty}\frac{1}{\varphi(s)} \int_0^{s} f^*(t) d t < \infty \right \}
$$
be the corresponding {\it Marcinkiewicz space}. It is known that $(M_\varphi, \|\cdot\|_{M_\varphi})$
is a fully symmetric space such that $\mathbf 1  \notin  M_\varphi$ if and only if
$\lim\limits_{t \to \infty}\frac{\varphi(t)}{t}=0$.

\section{On strong convergence of Ces\`aro averages}

In this section we give a characterization of fully symmetric spaces for which the mean ergodic theorem is valid.

Let $ (\Om,\cal A,\mu) $ be a  $\sigma$-finite measure space. If we consider the complete Boolean
algebra $\nabla_\mu = \{e=[E]: E \in \cal A\}$ of equivalence classes of $\mu$-a.e. equal sets in $\cal A$
 (that is, when $E,G\in \cal A$ and $\mu(E\Delta G)=0$), then
$\mu(e): = \mu(E)$ is a strictly positive  measure on $\nabla_\mu$. Denote by $\nabla_\nu(0, a) = \{[E]: E\in\cal A_\nu\}$ the complete Boolean algebra of equivalence classes of $\nu$-a.e. equal sets in $((0,a), \nu)$, $0< a\leq\ii$.

A  Boolean subalgebra $\nabla_0$ in $\nabla_\mu$ is called {\it regular} if $\sup D \in \nabla_0$ for every subset $D \subseteq \nabla_0$.  If $\nabla_0$ is a regular subalgebra in $\nabla_\mu$, then clearly
$\cal A_0 = \{E \in \cal A:[E] \in \nabla_0 \}$ is a $\sigma$-subalgebra in $\cal A$ and $\nabla_0 = \{[E]: E \in \cal A_0\}$.

It is known that
there exists $e \in \nabla_\mu$ such that $e\cdot \nabla_\mu$ is a non-atomic, that is, the Boolean algebra $e\cdot\nabla_\mu$ has not atoms, and
$(\mathbf 1-e)\cdot \nabla_\mu$ is a totally atomic Boolean algebra, that is, $\mb 1-e= \sup\limits_n q_n$, where
$\{q_n\}$ is the set of atoms in the Boolean algebra $\nabla_\mu$ (see, for example, \cite[I, Ch.2, \S 2]{vl}).

Let  $\nabla_\mu$ be a non-atomic Boolean algebra. In view of \cite[Ch.2, Corollary 7.6]{bs}, we have the following.

\begin{pro}\label{p41}
There exist a regular subalgebra $\nabla_0$ in $\nabla_\mu$ and a Boolean isomorphism
$\varphi: \nabla_\nu(0, \mu(\Omega)) \to \nabla_0$ onto such that
$\mu(\varphi(e)) = \nu(e)$ for all $e\in\nabla_\nu(0, \mu(\Omega))$.
\end{pro}

Utilizing Proposition \ref{p41} and \cite[Theorem 2.4]{cl}, we obtain the following.

\begin{cor}\label{c41}
Let $0\neq e\in\nabla_\mu$ be such that $e\cdot \nabla_\mu$ is non-atomic, and let $\nabla_0$ and $\varphi$ be as in
Proposition  \ref{p41} (with respect to the Boolean algebra $e\cdot \nabla_\mu$). Then there exists a unique algebraic  isomorphism $\Phi: L^0 ((0,\mu(\Om)),\nu) \rightarrow L^0 (\Om,\cal A_0, \mu)$ such that
\begin{enumerate}[(i)]
\item $\Phi(e)=\varphi(e)$ for all $e \in \nabla_\nu(0, \mu(\Om))$;
\item $\Phi: L^1((0,\mu(\Om)), \nu) \to  L^1(\Om,\cal A_0,\mu)$ and
$\Phi: L^\ii((0,\mu(\Om)), \nu)  \to L^\ii(\Om,\cal A_0, \mu)$ are bijective linear isometries.
\end{enumerate}
\end{cor}

In what follows, $T\in DS (\Om,\cal A,\mu)$ will mean that $T$ is a Dunford-Schwartz operator in
$L^1(\Om,\cal A,\mu)+L^\ii(\Om,\cal A,\mu)$.

If $E\in\cal A$ and $\cal A_E = \{A\cap E: A \in \cal A\}$, then it is clear that $(E, \mathcal A_E, \mu)$ is a $\sigma$-finite measure space. The next property of Dunford-Schwartz operators can be found in \cite[Corollary 2.1]{cl}.

\begin{teo}\label{t41}
Let $0 \neq e=[E] \in \nabla_\mu$, and let $\nabla_0$ be a regular subalgebra in $e\cdot \nabla_\mu$ such that
$(\Om,\cal A_0,\mu)$ is a $\sigma$-finite measure space.
If $T\in DS(E, \mathcal A_0, \mu)$, then there exists $\wh T\in DS(\Om,\cal A,\mu)$ such that
$$
\widehat T(g) = T(g) \text{ \ \ and \ \ } M_n(\wh T)(g) =  M_n(T)(g)
$$
for all $g \in L^1(E,\cal A_0,\mu)+L^{\ii}(E,\cal A_0,\mu)$ and $n\in\mathbb N$.
\end{teo}

We say that a fully symmetric space $X=X(\Om)$ possesses the {\it mean ergodic theorem property} (writing
$X\in(\text{MET})$) if the averages $M_n (T)$ converge strongly in $X$ for any $T\in DS(\Om,\cal A,\mu)$.

\begin{teo}\label{t42} If $(\Om,\cal A,\mu)$ is quasi-non-atomic, $\mu(\Om) = \ii$, and $X=(X(\Om), \|\cdot\|_X)$
is a fully symmetric space such that $X\su L^1(\Om)$, then $X \notin (\text{MET})$.
\end{teo}
\begin{proof}
Assume first $(\Om,\cal A,\mu)= ((0, \ii), \nu)$ and consider the operator $T \in DS((0,\ii),\nu)$ defined by
$T(f)(t)= f(t-1)$ if $t> 1$ and $T(f)(t) = 0$ if $t \in (0,1]$. Since
\begin{equation*}
\begin{split}
\|M_{2n}(T)(\chi_{(0,1]}) - M_{n}(T)(\chi_{(0,1]})\|_1 &= \left\|\frac1{2n}\chi_{(0,2n]}-\frac1{n}\chi_{(0,n]}\right\|_1\\
&=n\left (\frac1n-\frac1{2n}\right )+n\frac1{2n}=1,
\end{split}
\end{equation*}
it follows that  the averages $M_n(T)(\chi_{(0,1]}))$ do not converge in the norm $\|\cdot\|_1$.
Therefore $L^1((0, \ii),\nu)\notin(\text{MET})$.

It is known \cite[Ch.6, \S 6.1, Proposition 6.1.1]{rgmp} that an embedding $X_1 \subset X_2$ of symmetric spaces
$(X_1, \|\cdot\|_ {X_1}) $ and $(X_2, \|\cdot\|_ {X_2})$ is continuous, that is, there is a constant $c> 0$ such that
$ \| f \|_{X_2} \leq c \| f \|_{X_1} $ for all $f \in X_1$.
Since, by the assumption,  $X=X((0, \ii),\nu)\su L^1((0, \ii),\nu)$, it follows that
$\| f \|_1 \leq c \| f \|_X$ for all $f\in X$ and some constant $c >0$. Consequently, the sequence $\{M_n(T)(\chi_{(0,1]})\}$
can not converge strongly in the space $X$, and we conclude that $X((0,\ii),\nu)\notin (\text{MET})$.

Assume now that $(\Om,\cal A,\mu)$ is non-atomic. By Corollary \ref{c41},
there exist a regular subalgebra $\nabla_0$ in $\nabla_\mu$ and an algebraic  isomorphism $\Phi: L^0 ((0,\ii),\nu)\to L^0 (\Om,\cal A_0,\mu)$ such that $\mu(\Phi(e)) = \nu(e)$ for all $e\in\nabla_\nu(0, \mu(\Om))$ and $\Phi: L^1((0,\infty), \nu) \to  L^1(\Om,\cal A_0,\mu)$ and $\Phi: L^\ii((0,\ii), \nu)  \to L^\ii(\Om,\cal A_0, \mu)$ are bijective linear isometries.
Therefore $\wt T = \Phi\circ T\circ \Phi^{-1} \in DS(\Om,\cal A_0,\mu)$, and the sequence $\{M_n(\wt T)(\Phi(\chi_{(0,1]}))\}$
does not converge in the space $(X(\Om,\cal A_0,\mu), \|\cdot\|_{X(\Om,\cal A_0,\mu)})$. Note that $(\Om,\cal A_0,\mu)$ is a $\sigma$-finite measure space.

By Theorem \ref{t41}, there exists $\wh T\in DS(\Om,\cal A,\mu)$ such that $\wh T(g) = \wt T(g)$ and
$M_n(\wh T)(g) =  M_n(\wt T)(g)$ for all $g\in L^1(\Om,\cal A_0, \mu)+L^\ii(\Om,\cal A_0,\mu)$ and $n\in\mathbb N$.
Thus, the sequence $\{M_n(\wh T)(\Phi(\chi_{(0,1]}))\}$ does not converge in the space
$X=X(\Om,\cal A,\mu)$, hence $X\notin(\text{MET})$.

Next, let $(\Om,\cal A,\mu)$ be a totally atomic infinite measure space with the atoms of equal measures.
In this case $L^1(\Om) = l^1, \ L^\ii(\Om)= l^\ii$ and $l^1 \su X(\Om) \su l^\ii$, which, by the assumption, implies that
$X=X(\Om) =l^1 $. Define $T\in DS(\Om,\cal A,\mu)$ by $T(\{\xi_n\}_{n = 1}^\ii) = \{0,\xi_1,\xi_2, \dots \}$ if
$\{\xi_n \}_{n = 1}^\ii \in l^\ii$. If $e_1=\{1,0,0,\dots \}$, then we have
$$
\|M_{2n}(T)(e_1) - M_n(T)(e_1) \|_1 = $$
$$=\left\|\frac1{2n}\{\underbrace{1,1,\dots,1}_{2n},0,0,\dots\}-\frac1n\{\underbrace
{1,1,\dots,1}_n,0,0,\dots\}
\right\|_1=1,
$$
implying that the sequence  $\{M_n(T)(e_1)\}$ does not converge in the norm $\|\cdot\|_1$. Since $X=l^1 $ and
$\| f \|_X \leq \| f\|_1 $, $f \in l^1$, it follows that the norms $\| \cdot \|_X$ and $\| \cdot \|_1$ are  equivalent. Consequently, the sequence $\{M_n(T)(e_1)\}$ is not convergent in the norm $\|\cdot\|_X$ also,
hence $X\notin(\text{MET})$.

Assume now that $(\Om,\cal A,\mu)$ is an arbitrary quasi-non-atomic infinite measure space. As noted above, there exists
$e =[E]\in\nabla_\mu$ such that $e\cdot \nabla_\mu$ is a non-atomic and
$(\mb 1-e)\cdot \nabla_\mu$ is a totally atomic Boolean algebra.

Let $\mu(E) <\ii$. Since $\mu(\Om) = \ii$, it follows that $(\Om\sm E,\cal A_{\Om\sm E},\mu)$ is a totally atomic
$\sigma$-finite infinite measure space with the atoms of the same measure. According to what has been proven above,
we have $(\Om\sm E,\cal A_{\Om\sm E},\mu), \| \cdot \|_X)\notin (\text{MET})$.
Further, by Theorem \ref{t41}, there exists $\wh T\in DS(\Om,\cal A,\mu)$ such that
$\wh T(g) = T(g)$ and $ M_n(\wh T)(g) =  M_n(T)(g)$ for all $g\in L^1(\Om\sm E,\cal A_{\Om\sm E},\mu)+
L^\ii(\Om\sm E,\cal A_{\Om\sm E},\mu)$ and $n \in \mathbb N$. Therefore $X(\Om,\cal A,\mu) \notin (\text{MET})$.

If $\mu(E)=\ii$, then, as we have shown, $(X(E,\cal A_E, \mu), \| \cdot \|_X)  \notin (MET)$. In particular, there exist
$T \in DS(E,\cal A_E, \mu)$ and $f\in X(E,\cal A_E, \mu)$ such that the sequence $\{M_n(T)(f)\}$ is not convergent in the norm  $\|\cdot\|_X$. By Theorem \ref{t41}, there exists $\wh T\in DS(\Om,\cal A,\mu)$ such that
$\wh T(g) = T(g)$ and $M_n(\wh T)(g)=M_n(T)(g)$ for all $g \in L^1(E,\cal A_E, \mu)+L^\ii(E,\cal A_E,\mu) $ and
$n \in \mathbb N$. Therefore $X(\Om,\cal A,\mu)\notin (\text{MET})$.
\end{proof}

Recall that a fully symmetric space $(X, \|\cdot\|_X)$ on $(\Om,\cal A,\mu)$, which is generated by a  fully symmetric space
$(X(0,\ii),\| \cdot \|_{X(0,\ii)})$, has order-continuous norm if and only if the space
$(X(0,\ii),\| \cdot \|_{X(0,\ii)})$ is separable.

The next theorem gives another condition under which a fully symmetric space $X(\Om,\cal A,\mu)$ does not belong to
$(\text{MET})$.

\begin{teo}\label{t43}
Let $(\Om,\cal A,\mu)$ be a quasi-non-atomic $\sigma$-finite infinite measure space. If $X=(X(\Om), \|\cdot\|_X)$ is a fully symmetric space generated by a non-separable fully symmetric space $(X(0,\ii),\| \cdot \|_{X(0,\ii)})$,
then $X\notin (\text{MET})$.
\end{teo}
\begin{proof}
Assume first that $(\Om,\cal A,\mu)= ((0,\ii),\nu)$. Since $(X(0,\ii),\|\cdot\|_{X(0, \infty)})$ is not separable,
it follows that there exists $a>0$, such that the symmetric space $(X(0,a), \|\cdot\|_{X(0, a)})$
also is not separable \cite[Ch.II, \S 4, Theorem 4.8]{kps}. Therefore, by \cite[Theorem 2.5.1]{vf}, there are a function
$f_0 \in X(0, a)\sm\ol{L^\ii(0, a)}^{\|\cdot\|_{X(0, a)}}$ and a Dunford-Schwartz operator $T_0\in DS((0,a),\nu)$ such that the averages $M_{n}(T_0)(f_0)$ do not converge in the norm $\|\cdot\|_{X(0, a)}$.

Define a Dunford-Schwartz operator $T\in DS((0,\ii),\nu)$ by
$$
T(g) = T_0(g\cdot\chi_{(0, a)}), \ \ g\in L^1(0,\ii)+L^\ii(0,\ii).
$$
If we set $f = f_0\cdot \chi_{(0, a)}+0 \cdot\chi_{[a,\ii)} $, then $f \in X(0, \infty)$ and $M_n(T)(f)=M_n(T_0)(f_0)$
for every $n$. Consequently, the sequence $\{M_n(T)(f)\}$ does not converge strongly in $X(0, \infty)$.

Next, let $(\Om,\cal A,\mu)$ be non-atomic. By Corollary \ref{c41}, there exist a regular subalgebra
$\nabla_0$ in $\nabla_\mu$ and an algebraic isomorphism $\Phi: L^0 ((0,\ii),\nu) \to L^0(\Om,\cal A_0,\mu)$ such that
$\mu(\Phi(e)) = \nu(e)$ for all $e\in\nabla_\nu(0, \mu(\Om))$ and $\Phi: L^1((0,\infty), \nu) \to  L^1(\Omega, \mathcal A_0, \mu)$ \ and \ $\Phi: L^\ii((0,\infty), \nu)  \to L^\ii(\Omega, \mathcal A_0, \mu)$ are bijective linear isometries. In particular,  $(\Om,\cal A_0,\mu)$ is $\sigma$-finite. According to what has been proved above, there exists
$T\in DS((0,\ii),\nu)$ such that the averages $M_n(T)$ do not converge strongly in
$X(0,\ii)$. It is clear then that $\wt T = \Phi\circ T\circ \Phi^{-1} \in DS(\Om,\cal A_0,\mu)$ and the averages
$M_n(\wt T)$ do not converge strongly in $X(\Om,\cal A_0,\mu)$.

By Theorem \ref{t41}, there exists $\wh T\in DS(\Om,\cal A,\mu)$ such that
$\wh T(g)=\wt T(g)$ and $M_n(\wh T)(g)=M_n(\wt T)(g)$ for all $g\in L^1(\Om,\cal A_0,\mu)+L^\ii(\Om,\cal A_0, \mu) $ and $n \in \mathbb N$. It follows then that the averages $M_n(\wh T)$ do not converge strongly in
$X=X(\Om,\cal A,\mu)$, hence $X\notin (\text{MET})$.

Now, let $(\Om,\cal A,\mu)$ be totally atomic infinite measure space with the atoms of equal measures.
In this case $l^1 \subseteq X=X(\Om) \subseteq l^\ii$, and $\cal R_\mu = c_0$, the fully symmetric space of converging
to zero sequences $f=\{\xi_n\}_{n = 1}^\ii$ of real numbers with respect to the norm
$\|f\|_\ii=\sup\limits_{n \in \mathbb N} |\xi_n|$.

If  there is $f\in X\sm c_0$, then $f^*\ge\al \mathbf1$ for some $\al>0$, where $\mb 1 = \{1,1,\dots\}$, hence
$\mb 1 \in X$ and $X=l^\ii$. Therefore, if $X$ is a symmetric sequence space, then either $X\su c_0$ or
$X=l^\ii$. Since a.u. convergence in $l^\ii$ and $c_0$ coincides with the convergence in the norm $\|\cdot\|_\ii$, Theorem \ref{t34} implies that $l^\ii\notin (\text{MET})$.

Let now $X \su c_0$. Since, by \cite[Ch.II, Section 3.6]{lsz}, a symmetric space $X(\Om) \su c_0$ is separable if and only if the symmetric space $X(0,\ii)$ is separable, but $X(0,\ii)$ is not separable, it follows that $X$ is not a separable space.
Therefore, the norm $\|\cdot\|_X$  is not order continuous, implying that there exists
$f=\{\xi_n\}_{n = 1}^\ii =\{ \xi^*_n\}_{n = 1}^\ii\in X$ such that
\begin{equation}\label{e7}
 \xi_n  \downarrow  0 \text{ \ \ and \ \ } \|\{\underbrace{0,0,\dots,0}_n,\xi_{n+1},\xi_{n+2},\dots\}\|_X \downarrow\al >0.
\end{equation}
Let $T \in DS(\Om,\cal A,\mu)$ be defined as $T(\{\eta_n\}_{n = 1}^\ii) = \{0,\eta_1,\eta_2, \dots \}$ whenever
$\{\eta_n \}\in l^\ii$. Then $T^k(f) = \{\underbrace{0,0,\dots,0}_k,\xi_1,\xi_2,\dots\}$, so
\begin{equation*}
\begin{split}
\{\eta_m^{(n)}\}_{m=1}^\ii:=\sum_{k=0}^{n-1}T^k&(f)
=\{\xi_1,\xi_1+\xi_2,\dots,\xi_1+\xi_2+\dots\xi_n,\xi_2+\xi_3+\dots+\xi_{n+1},\\
&\xi_3+\xi_4+\dots+ \xi_{n+2}, \dots ,\xi_{m-n+1}+\xi_{m-n+2}+\dots+ \xi_m,\dots\},
\end{split}
\end{equation*}
that is,
$$
\eta_m^{(n)}=\xi_{m-n+1}+\xi_{m-n+2}+\dots+\xi_m \text{ \ \ if \ \ } m\ge n
$$
and
$$ \eta_m^{(n)}  = \xi_1+\xi_2+\dots\xi_m \text{ \ \ if \ \ } 1 \leq m< n.
$$
Since $\xi_n\downarrow  0 $, it follows that $\frac1n \sum\limits_{k=0}^{n-1}\xi_k \to 0$ as $n\to\ii$. Consequently,
$$0 \leq \frac1n \ \eta_m^{(n)}= \frac1n \ (\xi_{m-n+1}+\xi_{m-n+2}+\dots+ \xi_{m}) \leq
\frac1n \sum_{k=0}^{m}\xi_k \to 0
$$
as $n\to \ii$ for any fixed $m\in \mathbb N$. Therefore, the sequence $\{M_n(T)(f)\}$ converges to zero coordinate-wise.

Suppose that there exists $\wh f\in X$ such that $\|M_n(T)(f)-\wh f\|_X \to 0$.
Then the sequence $\{M_n(T)(f)\}$ converges to $\wh f$ coordinate-wise, implying that $\wh f= 0$.

On the other hand, as $\xi_n\downarrow  0$, we have
\begin{equation*}
\begin{split}
M_n(T)(f)=\bigg \{&\frac1n\xi_1,\frac1n(\xi_1+\xi_2),\dots,\frac1n(\xi_1+\xi_2+\dots+\xi_n),\frac1n(\xi_2+\xi_3+\dots+ \xi_{n+1}), \\
&\frac1n(\xi_3+\xi_4+\dots+\xi_{n+2}),\dots, \frac1n(\xi_{m-n+1}+\xi_{m-n+2}+\dots+ \xi_m),\dots\bigg\} \\
&\ge \{\underbrace{0,0,\dots,0}_n,\xi_{n+1},\xi_{n+2},\dots\} \geq 0.
\end{split}
\end{equation*}
Thus, by (\ref{e7}),
$$
\|M_n(T)(f)\|_X \ge \|\{\underbrace{0,0,\dots,0}_n,\xi_{n+1},\xi_{n+2},\dots\}\|_X \ge\al >0,
$$
implying that the sequence $\{M_n(T)(f)\}$ is not convergent with in the norm $\|\cdot\|_X$, that is, $X\notin (\text{MET})$.

Repeating the ending of the proof of Theorem \ref{t42}, we conclude that $X(\Om)\notin (\text{MET})$ for any quasi-non-atomic $\sigma$-finite infinite measure space $(\Om,\cal A,\mu)$.

\end{proof}

Let $(X, \|\cdot\|_X )$ be a symmetric space on $((0,\ii),\nu)$. The {\it fundamental function} of $X $ is defined by
$\varphi_X(t) = \|\chi_{(0,t]}\|_X$. It is known that $\varphi_X(t)$ is a {\it quasi-concave} function (see \cite[Ch.II, \S 4,
Theorem 4.7]{kps}); in particular, $\varphi_X(t)$ increases, while
the function $\frac{\varphi_X(t)}t$ decreases \cite[Ch.II, \S 1, Definition 1.1]{kps}. Consequently, there exist the limits
$$
\al(X)= \lim\limits_{t\to\ii}\frac{\varphi_X(t)}t \text{ \ \ and \ \ } \bt(X) = \lim\limits_{t \to 0^+}\varphi_X(t)
=\varphi_X(+0).
$$
Note that
$$
\al (L^1)= 1,\ \ \bt (L^\ii)= 1,\ \ \al(L^p)=0, \ 1 < p \leq\ii,\ \ \bt(L^p)= 0, \ 1 \leq p <\ii.
$$
We need the following necessary and sufficient conditions for an embedding of a symmetric space $X(\Om)$
into $L^\ii(\Om)$ or $L^1(\Om)$.

\begin{pro}\label{p42}
Let $(\Om,\cal A,\mu)$ be a $\sigma$-finite infinite measure space.
If $X=X(\Om)$ is the symmetric space generated by a symmetric space $(X(0,\ii), \|\cdot\|_X)$ on $((0,\ii),\nu)$, then
\begin{enumerate}[(i)]
\item
$X\subseteq L^\ii(\Om)$ if and only if $\bt(X) >0$;
\item
$X\subseteq L^1(\Om)$ if and only if $\al(X) >0$;
\item
$X\subseteq L^1(\Om)$ if and only if $L^\ii(\Om)\subseteq X^\times$, where $X^\times$ is the K\"othe dual of $X$.
\end{enumerate}
\end{pro}
\begin{proof}
It is clear that $X(\Om) \subseteq Y(\Om)$ if and only if $X(0,\ii) \subseteq  Y(0,\ii) $, where $Y(0,\ii)$ is a symmetric  space on $((0, \infty), \nu)$ that generated $Y(\Om)$. Consequently, it is sufficient to prove the proposition in the case
$(\Om,\cal A,\mu) = ((0,\ii),\nu)$.

(i) If $X=X(0,\ii) \subseteq L^\ii(0,\ii)$, then there is $c_0 >0$ such that $\| f \|_\ii \leq c_0 \| f \|_X$
for all $f\in X$ \ \cite[Ch.6, \S 6.1, Proposition 6.1.1]{rgmp}. Therefore
$$
\varphi_X(t) = \|\chi_{(0,t]} \|_{X} \ge\frac {\| \chi_{(0,t]} \|_\ii}{c_0} =\frac1{c_0}
$$
and
$$
\bt (X) = \lim\limits_{t \to 0^+}\varphi_X(t) \ge\frac1{c_0}>0.
$$
If $X\nsubseteq L^\ii(0,\ii) $, then there is a positive unbounded function
$f$ in $X(0,\ii)\sm L^\ii(0,\ii)$; in particular, $\nu(A_n) >0$, where $A_n=\{f\ge n\}$, $n\in\mathbb N$.
Choose a sequence $B_n \subseteq A_n$ such that $B_n\supseteq B_{n + 1}$, $0<\nu(B_n)<\ii$, and
$\lim\limits_{n\to\ii} \nu(B_n)=0$. Then we have
$$
n\cdot\bt(X) \leq n \| \chi_{B_n}\|_X=\| n\cdot\chi_{B_n}\|_X\leq\| f \|_X<\ii \ \ \ \forall\ \ n \in \mathbb N,
$$
hence $\bt(X)=0$.

(ii) If  $X\subseteq L^1(0,\ii)$, then there is $c_1> 0$ such that $\| f \|_1\leq c_1\| f \|_X$ for all $f\in X$
\cite[Ch.6, \S 6.1, Proposition 6.1.1]{rgmp}.
Consequently,
$$
\frac{\varphi_X(t)}t = \frac{\| \chi_{(0,t]} \|_X}t\ge\frac{\| \chi_{(0,t]}\|_1}{c_1\cdot t}=\frac1{c_1}
$$
and
$$
\al (X) = \lim\limits_{t\to\ii}\frac{\varphi_X(t)}t \geq\frac1{c_1} > 0.
$$
Assume now that $\al(X)>0$. By \cite[Ch.II, \S 4, Inequality (4.6)]{kps}, we have
$$
\|f^*\cdot \chi_{(0,t]} \|_1 \leq\frac1{\varphi_X(t)}\cdot \|f^*\cdot\chi_{(0,t]}\|_X
$$
for all $f\in X$. Since
$\frac{\varphi_X(t)}t\geq\al(X) > 0$, $t>0$, it follows that $\frac t{\varphi_X(t)} \leq\frac1{\al(X)}$, implying that
$$
\|f^* \cdot \chi_{(0,t]}\|_1\leq\frac{\|f^* \cdot \chi_{(0,t]}\|_X}{\al(X)} \leqslant \frac{\|f^*\|_X}{\al (X)},
$$
so $\|f\|_1 \leq\frac{ \|f\|_X}{\al(X)}<\ii$ for all $f\in X$, that is,  $X\subseteq L^1(0,\ii)$.

(iii) If $X\subseteq  L^1(0,\ii)$, then
$\int\limits_0^\ii|\mb 1 \cdot f| d \nu=\|f\|_1<\infty$ for all $f\in X$. Thus
$\mb 1 \in X^\times$ and $L^\ii(0,\ii) \subseteq X^\times$.

Conversely, if $L^\ii(0,\ii) \subseteq X^\times$, then $\mb 1 \in  X^\times$, that is, the linear functional
$$
\varphi (f)= \int\limits_0^\ii\mb 1\cdot f \ d \nu, \ \ f \in X,
$$
is bounded on $X$.Then it follows that
$$
\|f\|_1 = \int\limits_0^\ii\mb 1 \cdot |f| d\mu = \varphi (|f|) \leq\|\varphi\| \cdot \|f\|_X<\ii
$$
for every function $f\in X$, hence $X\subseteq  L^1(0,\ii)$.
\end{proof}

\begin{cor}\label{c42}
Let $(\Om,\cal A,\mu)$ and $X(0,\ii)$ be as in Proposition \ref{p42}. Then following are equivalent
\begin{enumerate}[(i)]
\item
$X(\Om)\nsubseteq L^1(\Om)$;
\item
$L^\ii(\Om) \nsubseteq X^\times(\Om)$;
\item
$\mb 1 \notin X^\times(\Om)$.
\end{enumerate}
\end{cor}

Let $(\Om,\cal A,\mu)$ be a $\sigma$-finite infinite measure space. Let $f\in\cal R_\mu$ and $T\in DS(\Om,\cal A,\mu)$.
By Theorem \ref{t31}, there exists $\wh f\in\cal R_\mu$ such that the sequence $\{M_n(T)(f)\}$ converges a.u. to $\wh f$.
Define the mapping $P:\cal R_\mu\to\cal R_\mu $ by setting
$$
P(f)=\wh f=(a.u.)-\lim_{n\to\ii} M_n(T)(f), \ \ f \in\cal R_\mu.
$$
It is clear that $P$ is linear. Since  $(L^1,\|\cdot\|_1)$ possesses Fatou property and $\|M_n(T)(f)\|_1 \leq \|f\|_1$ for all 
$f \in L^1$, it follows that $\| P(f)\|_1 \leq \|f\|_1$ for every $f\in L^1$ \cite[Ch.IV, \S 3, Lemma 5]{ka}, that is, 
$\| P\|_{L^1\to L^1} \leq 1$.

Similarly, if $f\in L^1\cap L^\ii$, then $\|M_n(T)(f)\|_\infty \leqslant \|f\|_\infty$.
Therefore, a.u. convergence $M_n(T)(f)\to P(f)$ implies that $\| P(f)\|_\ii \leq |f\|_\ii$.

According to \cite[Theorem 3.1]{ccl}, there exists a unique operator $\wh P \in DS$  such that $\widehat{P}(f)=P(f)$   for all $f \in  \mathcal R_\mu$; in particular, $\|P\|_{\cal R_\mu \to \cal R_\mu} \leq 1$.

Additionally, by the classical mean ergodic theorem in the space $ L^2$, we have
$\|M_n(T)(f)-P(f)\|_2\to 0$ as $n\to\ii$ for any $f\in L^2$.

The next theorem is a version of the mean ergodic theorem for the fully symmetric space
$(\cal R_\mu,\|\cdot\|_{L^1+L^\ii})$.

\begin{teo}\label{t44}
If $T\in DS$, then
\begin{equation}\label{e71}
\|M_n(T)(f)-P(f)\|_{L^1+L^\ii}\to 0 \ \ \text{for all} \ \ f\in\cal R_\mu.
\end{equation}
\end{teo}
\begin{proof}
We have
$$
\sup\limits_n\|M_n(T)\|_{L^1+L^\ii\to L^1+L^\ii}\leq 1 \text{ \ \ and \ \ } \| P\|_{\cal R_\mu\to\cal R_\mu} \leq 1.
$$
Since $\|M_n(T)(f)-P(f)\|_2 \to 0$, it follows that $\|M_n(T)(f)-P(f)\|_{L^1+L^\ii}\to 0$ for any $f \in L^1\cap L^\ii\su L^2$. Using density of $L^1 \cap L^\ii$ in the Banach space $(\cal R_\mu, \|\cdot\|_{L^1+L^\ii})$ and the principle  of uniform boundedness, we arrive at (\ref{e71}).
\end{proof}

Now we can establish the following important property of the operator $P$.
\begin{pro}\label{p43}
$P^2=P$ and
$$
TP(f)=P(f)=PT(f)\text{\ \ for all\ \ } f\in\cal R_\mu.
$$
\end{pro}
\begin{proof}
Since
$$
(I-T)M_n(T)=\frac{I-T^n}n=\frac In+M_n(T)-\frac{n+1}nM_{n+1}(T),
$$
it follows that
$$
P(f)-PT(f)=a.u.-\lim_{n\to\ii}(I-T)M_n(T)(f)=0,
$$
hence $PT(f) = P(f)$, for all $f\in\cal R_\mu$.

Denote $\|\cdot\|_{L^1+L^\ii}$ by $\|\cdot\|$. Then, By Theorem \ref{t44}, $\|M_n(T)(f)-P(f)\|\to 0$
for each $f\in\cal R_\mu$. Consequently,
\begin{equation*}
\begin{split}
TP(f) &= T\left( \|\cdot\|-\lim_{n\to\ii} M_n(T)(f)\right)\\
&=\|\cdot\|-\lim_{n\to\ii} \frac1n\sum_{k=1}^n T^k(f)
=\|\cdot\|-\lim_{n\to\ii} \left(\frac1n\sum_{k=0}^n T^k(f)-\frac fn\right)\\
&=\|\cdot\|-\lim_{n\to\ii} \left(\frac{n+1}nM_{n+1}(T)(f)-\frac fn\right)=P(f).
\end{split}
\end{equation*}
Therefore $TP(f) = P(f) = PT(f)$, $f\in\cal R_\mu$, hence $M_n(T)P = P$, $n=1,2,\dots$, implying that $P^2=P$.
\end{proof}

We will also need the following property of symmetric spaces \cite[Proposition 2.2]{dds}.
\begin{pro}\label{p44}  Let $(\Om,\cal A,\mu)$ be a $\sigma$-finite infinite measure space.
Let $(X,\|\cdot\|_X)$ be a separable symmetric space on $((0,\ii),\nu)$ such that
$X^\times\subseteq\cal R_\nu$. If $\{f_n\}\su X(\Om)$ and $g\in X(\Om)$ are such that $f_n \prec\prec g$ for all $n$,
then $f_n\to 0$ in measure implies that $\|f_n\|_{X(\Om)}\to 0$ as $n\to\ii$.
\end{pro}

\begin{teo}\label{t45}
Let $X=(X(\Om), \|\cdot\|_X)$ be a fully symmetric space on a $\sigma$-finite measure space
$(\Om,\cal A,\mu)$. If the norm $\|\cdot\|_X$ is order continuous and  $L^1\nsubseteq X$, then the averages $M_n(T)$ converge strongly in $X$ for each $T\in DS$.

If $(\Om,\cal A,\mu)$ is quasi-non-atomic, then
strong convergence of the averages $M_n(T)$ for every $T\in DS$ implies that the norm
$\|\cdot\|_X$ is order continuous and $L^1\nsubseteq X$.
\end{teo}
\begin{proof}
Since the symmetric space $(X(0,\ii),\| \cdot \|_{X(0,\ii)})$ is separable, it follows that $X(0,\ii)\subseteq\cal R_\nu$,
hence $X=X(\Om)\subseteq\cal R_\mu$. As $L^1\nsubseteq X$, we have $L^1(0,\ii) \nsubseteq X(0,\ii)$.
Therefore, by  Corollary \ref{c42}, $\mb1 \notin X^\times(0,\ii)$, so $X^\times(0,\ii) \subseteq\cal R_\nu$ by Proposition \ref{p21}.

Since $X$ is a fully symmetric space and $\wh P \in DS$,  it follows that  $g=f - P(f) \in X\su\cal R_\mu$ for any $f\in X$.  
By Proposition \ref{p43}, that $P(g)=0$. Therefore, in view of Theorem \ref{t34}, $M_n(T)(g) \to P(g)=0$ in measure. 
Since $M_n(T)(g) \prec\prec g \in X$ for every $n$, Proposition \ref{p44} entails that $\|M_n(T)(g)\|_X\to 0$.
Next, by Proposition \ref{p43},
$$
M_n(T)(g)=M_n(T)(f) - M_n(T)(P(f))=M_n(T)(f)-P(f),
$$
implying that $\|M_n(T)(f)-P (f)\|_X\to 0$.

If $(\Om,\cal A,\mu)$ is quasi-non-atomic and the averages $M_n(T)$ converge strongly for every $T\in DS$, then 
Theorems \ref{t42} and \ref{t43} entail that the norm $\| \cdot \|_X$ is order continuous and $L^1\nsubseteq X$.
\end{proof}

Utilizing Theorem \ref{t45} and Proposition \ref{p42}, we can now state the following.
\begin{cor}\label{c43}
Let $(X, \|\cdot\|_X)$ be a fully symmetric space on a $\sigma$-finite measure space $(\Om,\cal A,\mu)$.
If the norm   $\|\cdot\|_X$ is order continuous and $\al(X) =0$, then the averages $M_n(T)$ converge strongly in $X$ for each $T\in DS$.

If $(\Om,\cal A,\mu)$ is quasi-non-atomic, then strong convergence of the averages averages $M_n(T)$ for every
$T\in DS$ implies that the norm $\|\cdot\|_X$ is order continuous and $\al (X) =0$.
\end{cor}

\section{Ergodic theorems in Orlicz, Lorentz and Marcinkiewicz  spaces}

In this section we give applications of Theorems \ref{t31},  \ref{t34}, and \ref{t45}, to Orlicz, Lorentz
and Marcinkiewicz spaces.

1. Let $\Phi$ be an Orlicz function, and let $L^\Phi=(L^\Phi(\Om),\| \cdot \|_{\Phi})$ be corresponding Orlicz  space.
As noted in Section 3, if $\Phi(u)>0$ for all $u\ne 0$, then $\mb1\notin L^\Phi$; if $\Phi(u)=0$ for all $0\leq u< u_0$,
then $\mb1\in L^\Phi$.

Therefore, Theorems \ref{t31}  and \ref{t34} imply the following

\begin{teo}\label{t51}
Let $(\Om,\cal A,\mu)$ be an arbitrary measure space, and let $\Phi$ be an Orlicz function. If $\Phi(u)>0$ for all $u>0$,
$T\in DS$, and $f\in L^\Phi$, then there exists $\wh f\in L^\Phi$ such that  the averages (\ref{e11}) converge a.u. to
$\wh f$.

If \ $(\Om,\cal A,\mu)$ is quasi-non-atomic and $\Phi(u)=0$ for all $0\leq u< u_0$,
then there exist $T\in DS$ and $f \in  L^\Phi$ such that the averages $M_n(T)(f)$ do not converge a.e., hence a.u.
\end{teo}

It is said that an Orlicz function $\Phi$ satisfies {\it $(\Delta_2)$-condition at $0$ (at $\infty$)} if there exist $u_0\in (0,\ii)$ and $k >0 $  such that $\Phi(2u) < k\cdot \Phi(u)$ for all $0<u< u_0$ (respectively, $u >u_0$).
An Orlicz function $\Phi$ satisfies $(\Delta_2)$-condition at $0$ and at $\infty$ if and only if $(L^\Phi(0,\ii), \| \cdot\|_\Phi)$ has order continuous norm \cite[Ch.2, \S 2.1, Theorem 2.1.17]{es}.

By \cite[Ch.2, \S 2.2, Theorem 2.2.3]{es}, $L^\Phi(0,\ii) \subseteq L^1(0,\ii)$ if and only if
$$
\limsup\limits_{u\to 0}\frac{\Phi(u)}u>0\text{ \ \ and \ \ }\limsup\limits_{u\to \ii}\frac{\Phi(u)}u=0.
$$
Therefore, Theorem \ref{t45} yields the following.

\begin{teo}\label{t52}
Let $(\Om,\cal A,\mu)$ be $\sigma$-finite measure space, and let an Orlicz function $\Phi$ satisfies $(\Delta_2)$-condition at $0$ and at $\infty$. If $\lim\limits_{u\to 0}\frac{\Phi(u)}u =0$ or $\limsup\limits_{u \to \ii}\frac{\Phi(u)}u>0$, then the averages $M_n(T)$ converge strongly in $L^\Phi$ for all $T\in DS$.

If $(\Om,\cal A,\mu)$ is quasi-non-atomic measure space, then strong convergence of the averages $M_n(T)$ for every $T\in DS$ implies that the Orlicz function $\Phi$ satisfies $(\Delta_2)$-condition at $0$ and at $\ii$; in addition,
$\lim\limits_{u \to 0} \frac{\Phi(u)}u=0$ or $\limsup\limits_{u \to \ii}\frac{\Phi(u)}u>0$.
\end{teo}

2. Let $\varphi$ be a concave function on $[0,\ii)$ with $\varphi(0) = 0$ and $\varphi(t)>0$ for all $t > 0$,
and let $\Lb_\varphi=(\Lb_\varphi(\Om),\| \cdot \|_{\Lb_\varphi})$ be the corresponding Lorentz space.
As noted in Section 3, $\varphi(\ii) = \ii$ if and only if $\mb1\notin\Lb_\varphi$.
Therefore, Theorems \ref{t31} and \ref{t34} imply the following.

\begin{teo}\label{t53}
If $(\Om,\cal A,\mu)$ is an arbitrary measure space and $\varphi(\ii) = \ii$, then for all $T\in DS$ and  $f\in\Lb_\varphi$ there exists $\wh f\in\Lb_\varphi$ such that the averages (\ref{e11}) converge a.u. to  $\wh f$.

If $(\Om,\cal A,\mu)$ quasi-non-atomic and $\varphi(\ii) < \ii$, then there exist $T\in DS$ and $f\in\Lb_\varphi$ such that  the averages $M_n(T)(f)$ do not converge a.e., hence a.u.
\end{teo}

It is well-known that the space $(\Lb_\varphi(0, \infty),\| \cdot \|_{\Lb_\varphi})$ is separable if and only if $\varphi(+0) = 0$ and $\varphi(\ii) = \ii$ (see, for example, \cite[Ch.II, \S 5, Lemma 5.1]{kps}, \cite[Ch.9, \S 9.3, Theorem 9.3.1]{rgmp}). In addition, the fundamental function $\varphi_{\Lb_\varphi}(t) = \varphi(t)$. Therefore, Corollary \ref{c43} entail the following.

\begin{teo}\label{t54}
Let $(\Om,\cal A,\mu)$ be $\sigma$-finite, and let $\varphi$ be a concave function on $[0, \infty)$ with $\varphi(0) = 0$,
and $\varphi(t) > 0$ for all $t > 0$. If $\varphi(+0) = 0$, $\varphi(\ii)=\ii$, and $\al(\Lb_\varphi)=\lim\limits_{t\to\ii}\frac{\varphi(t)}t=0$, then the averages $M_n(T)$ converge strongly in $\Lambda_\varphi$ for each $T\in DS$.

If $(\Om,\cal A,\mu)$ is quasi-non-atomic, then strong convergence of averages $M_n(T)$ for every $T\in DS$ implies that
$\varphi$ satisfies conditions  $\varphi(+0) = 0$, $\varphi(\ii) = \ii$, and  $\lim\limits_{t\to\ii}\frac{\varphi(t)}t=0$.
\end{teo}

3. Let $\varphi$ be as above, and let $M_\varphi=(M_\varphi(\Om),\|\cdot\|_{M_\varphi})$ be the corresponding Marcinkiewicz space. As noted in Section 3, $\mb1\notin M_\varphi$ if and only if $\lim\limits_{t \to \ii}\frac{\varphi(t)}t=0$. Thus, the corresponding version of Theorem \ref{t53} holds for Marcinkiewicz space $M_\varphi$ if we replace condition
$\varphi(\ii)= \ii$  by $\lim\limits_{t\to\ii}\frac{\varphi(t)}t=0$.

If $\varphi(+0)>0$ and $\varphi(\ii) < \ii$, then $M_\varphi= L^1$ as the sets. In this case, if $(\Om,\cal A,\mu)$ is quasi-non-atomic, it follows from Theorem \ref{t42} that $M_\varphi\notin (\text{MET})$.

Let $\varphi(+0) = 0$ and $\varphi(\ii) = \ii$. If  $\lim\limits_{t \to 0}\frac{\varphi(t)}t=\ii$, then $M_\varphi$ is non-separable \cite[Ch. II \S 5, Lemma 5.4]{kps}. Consequently, if $(\Om,\cal A,\mu) $ is quasi-non-atomic, then, by Theorem \ref{t43}, $M_\varphi\notin (\text{MET})$.

\end{document}